\documentclass[11pt]{amsart}

\usepackage{palatino}
\usepackage{amsfonts}
\usepackage{amssymb}
\usepackage{amscd}
\usepackage{xypic}
\usepackage[breaklinks,bookmarksopen,bookmarksnumbered]{hyperref}
\usepackage{graphicx}
\usepackage{float}

\newtheorem{theorem}{Theorem}[section]
\newtheorem{proposition}[theorem]{Proposition}
\newtheorem{corollary}[theorem]{Corollary}
\newtheorem{lemma}[theorem]{Lemma}
\theoremstyle{definition}

\newtheorem{remark}[theorem]{Remark}

\newtheorem{conjecture/question}[theorem]{Conjecture/Question}

\newtheorem{remark/definition}[theorem]{Remark/Definition}
\newtheorem{terminology/notation}[theorem]{Terminology/Notation}

\def\GG{{\textbf G}}

\def\PP{{\textbf P}}
\def\FF{{\textbf F}}
\def\OO{\mathcal{O}}

\def\cA{\mathcal{A}}
\def\F{\mathcal{F}}

\def\E{\mathcal{E}}
\def\G{\mathcal{G}}
\def\X{\mathcal{X}}

\def\I{\mathcal{I}}

\def\cM{\mathcal{M}}

\def\cC{\mathcal{C}}

\def\mm{\overline{\mathcal{M}}}

\DeclareMathOperator{\codim}{codim}

\DeclareFontFamily{OT1}{pzc}{}
\DeclareFontShape{OT1}{pzc}{m}{it}{<-> s * [1.10] pzcmi7t}{}
\DeclareMathAlphabet{\mathpzc}{OT1}{pzc}{m}{it}

\begin{document}
\title{\small The universal K3 surface of genus 14 via cubic fourfolds}

\author[G. Farkas]{Gavril Farkas}

\address{Humboldt-Universit\"at zu Berlin, Institut f\"ur Mathematik,  Unter den Linden 6
\hfill \newline\texttt{}
 \indent 10099 Berlin, Germany} \email{{\tt farkas@math.hu-berlin.de}}
\thanks{}

\author[A. Verra]{Alessandro Verra}
\address{Universit\`a Roma Tre, Dipartimento di Matematica, Largo San Leonardo Murialdo \hfill
 \newline \indent 1-00146 Roma, Italy}
 \email{{\tt
verra@mat.uniroma3.it}}

\begin{abstract}
Using Hassett's isomorphism between the Noether-Lefschetz moduli space $\cC_{26}$ of special cubic fourfolds $X\subset \PP^5$ of discriminant $26$ and the moduli space
$\F_{14}$ of polarized $K3$ surfaces of genus $14$, we use the family of 3-nodal scrolls of degree seven in $X$ to show that the universal $K3$ surface over $\F_{14}$ is rational.
\end{abstract}

\maketitle

\section{Introduction}

For a very general cubic fourfold $X\subseteq \PP^5$, the lattice $A(X):=H^{2,2}(X)\cap H^4(X, \mathbb Z)$ of middle Hodge classes contains only classes of complete intersection surfaces, so
$A(X)=\langle h^2 \rangle$, where $h\in \mbox{Pic}(X)$ is the hyperplane class (see \cite{V}). Hassett, in his influential paper \cite{H1}, initiated the study of Noether-Lefschetz special cubic fourfolds.  If $\cC$ is the $20$-dimensional coarse moduli space of smooth cubic fourfolds $X\subseteq \PP^5$, let $\cC_d$ be the locus of \emph{special} cubic fourfolds $X$ characterized by the existence of an embedding of a saturated rank $2$ lattice
$$L:=\langle h^2, [S]\rangle \hookrightarrow A(X),$$
of discriminant $\mbox{disc}(L)=d$, where $S\subseteq X$ is an algebraic surface not homologous to a complete intersection. Hassett \cite{H1} showed that $\cC_d\subseteq \cC$ is an irreducible divisor, which is nonempty if and only if $d>6$ and $d\equiv 0,2 (\mbox{mod } 6)$. The study of the divisors $\cC_d$ for small $d$ has received considerable attention. For instance, $\cC_8$ consists of cubic fourfolds containing a plane, whereas $\cC_{14}$ corresponds to cubic fourfolds containing a quintic del Pezzo surface, see \cite{H2}. Relying on Fano's work \cite{Fa}, recently Bolognesi and Russo \cite{BR} have shown that all fourfolds $[X]\in \cC_{14}$ are rational.

\vskip 4pt

For every $[X] \in \cC$, we denote by $F(X):=\bigl\{\ell\in \GG(1,5):\ell\subseteq X\bigr\}$ the Hilbert scheme of the lines contained in $X$. It is well known \cite{BD} that $F(X)$ is a hyperk\"ahler fourfold deformation equivalent to the Hilbert square of a $K3$ surface. For discriminant $d=2(n^2+n+1)$, where $n\geq 2$, it is shown in \cite{H1} that $F(X)$ is \emph{isomorphic} to the Hilbert scheme $S^{[2]}$ of a polarized $K3$ surface $(S,H)$ with  $H^2=d$.
If $\F_g$ denotes the moduli space of polarized $K3$ surfaces of genus $g$, the previous assignment induces a rational map $$\F_{\frac{d}{2}+1}\dashrightarrow \cC_d,$$ which is a birational isomorphism for $d\equiv 2 (\mbox{mod } 6)$ and a degree $2$ cover for $d\equiv 0 (\mbox{mod } 6)$. This map, though non-explicit for it is defined at the level of moduli spaces of weight-$2$ Hodge structures, opens the way to the study of $\F_{n^2+n+2}$ via the concrete geometry of cubic fourfolds, without making a direct reference to $K3$ surfaces! The main result of this paper concerns the universal $K3$ surface $\F_{g,1}\rightarrow \F_g$.

\begin{theorem}\label{main1}
The universal $K3$ surface $\F_{14,1}$ of genus $14$ is rational.
\end{theorem}

Nuer \cite{Nu} proved that $\cC_{26}$ (and hence $\F_{14}$ as well) is unirational. His proof relies on the fact that a general fourfold $[X]\in \cC_{26}$ contains certain smooth rational surfaces, whose parameter space forms a unirational family. One can also show that $\cC_{44}$ is unirational, for a general $[X]\in \cC_{44}$ contains a Fano embedded Enriques surface and their moduli space is unirational, see \cite{Ve2} and also \cite{Nu}. Recently, Lai \cite{L} showed that $\cC_{42}$ is uniruled.

\vskip 3pt

Mukai in a celebrated series of papers \cite{M1}, \cite{M2}, \cite{M3}, \cite{M4}, \cite{M5}  established structure theorems for polarized $K3$ surfaces of genus $g\leq 12$, as well as $g=13,16,18,20$.  In particular, $\F_g$ is unirational for those value of $g$. No structure theorem for the general $K3$ surface of genus $14$ is known. A quick inspection of Mukai's methods shows that the universal $K3$ surface $\F_{g,1}$ is unirational for $g\leq 11$ as well. On the other hand, Gritsenko, Hulek and Sankaran \cite{GHS} have proved that $\F_g$ is a variety of general type for $g>62$, as well as for $g=47,51,53,55,58,59,61$. In a similar vein, recently it has been established in \cite{TVA} that $\cC_d$ is of general type for all $d$ sufficiently large. As pointed out in Remark \ref{19}, whenever $\F_g$ is of general type, the Kodaira dimension of $\F_{g,1}$ is equal to $19$.

\vskip 5pt

The proof of Theorem \ref{main1} relies on the connection between singular scrolls and special cubic fourfolds. We fix a general point $[X]\in \cC_{26}$ and denote by $S$ the
\emph{associated} $K3$ surface, such that $S^{[2]}\cong F(X)\hookrightarrow \GG(1,5)$. For each $p\in S$, we introduce the rational curve $$\Delta_p:=\bigl\{\xi\in S^{[2]}:\{p\}= \mathrm{supp}(\xi)\bigr\}.$$ Under the Pl\"ucker embedding $\GG(1,5)\subseteq \PP^{14}$, the degree of $\Delta_p\subseteq F(X)$ is equal to $7$, which suggests that each point of $p\in S$ parametrizes a \emph{septic } scroll $R=R_p\subseteq X$. Imposing the condition $\mbox{disc}\langle h^2, [R]\rangle =26$, one obtains $R^2=25$. Assuming $R$ has isolated non-normal nodal singularities, the double point formula implies that $R$ has precisely $3$ non-normal nodes. We shall prove that indeed, a general fourfold $[X]\in \cC_{26}$ carries a $2$-dimensional family of $3$-nodal scrolls $R\subseteq X$ with $\mbox{deg}(R)=7$. Furthermore, this family of scrolls is  parametrized by the $K3$ surface $S$ associated to $X$.

\vskip 4pt

We now describe the moduli space of $3$-nodal septic scrolls. We start with the Hirzebruch surface $\FF_1:=\mbox{Bl}_{\mathpzc{o}}(\PP^2)$, where $\mathpzc{o}\in \PP^2$, and denote by $\ell$ the class of a line and by $E$ the exceptional divisor. The smooth septic scroll $R'=S_{3,4}\subseteq \PP^8$ is the image of the linear system
$$\phi_{|4\ell-3E|}:\FF_1\hookrightarrow \PP^8.$$
We shall show in Section 3 that the secant variety $\mbox{Sec}(R')\subseteq \PP^8$ is as expected  $5$-dimensional. Choose general points $a_1, a_2, a_3\in \mbox{Sec}(R')$ and denote by
$\Lambda:=\langle a_1,a_2,a_3\rangle \in \GG(2,8)$ their linear span. The image $R\subseteq \PP^5$ of the projection with center $\Lambda$
$$\pi_{\Lambda}:R'\rightarrow \PP^5$$ is a $3$-nodal septic scroll. Conversely, up to the action of $PGL(6)$ on the ambient projective space $\PP^5$, each such scroll appears in this way.
We denote by
$\mathfrak{H}_{\mathrm{scr}}$ the moduli space of unparametrized $3$-nodal septic scrolls in $\PP^5$, that is,
the quotient of the corresponding Hilbert scheme under the action of $PGL(6)$. Then as showed in Proposition \ref{isom1}, the space $\mathfrak{H}_{\mathrm{scr}}$ turns out to be  birationally isomorphic to the $9$-dimensional unirational variety
$$\mathfrak{H}_{\mathrm{scr}}\cong \mbox{Sym}^3\bigl(\mbox{Sec}(R')\bigr)/ \mbox{Aut}(R').$$

Fix a general $3$-nodal septic scroll $R\subseteq \PP^5$.  A general  $X\in \PP \bigl(H^0(\mathcal{I}_{R/\PP^5}(3))\bigr)=\PP^{12}$ is a smooth cubic fourfold. Since $R$ has no further singularities apart from the three non-normal nodes, the double point formula implies that $[X]\in \cC_{26}$. One sets up the following incidence correspondence between scrolls and cubic fourfolds of discriminant $26$:
$$\xymatrix{
  & \mathfrak{X}:=\Bigl\{(X,R):R\subseteq X\Bigr\}\big / PGL(6) \ar[dl]_{\pi_1} \ar[dr]^{\pi_2} & \\
   \cC_{26} & & \mathfrak{H}_{\mathrm{scr}}       \\
                 }$$
Thus $\mathfrak{X}$ is birational to a $\PP^{12}$-bundle over the unirational variety $\mathfrak{H}_{\mathrm{scr}}$. We then show that the fibre over a general cubic fourfold $[X]\in \cC_{26}$ of the projection $\pi_1$ is $2$-dimensional and isomorphic to the $K3$ surface $S$ appearing in the identification $F(X)\cong S^{[2]}$. We summarize the discussion above.

\begin{theorem}\label{univk3}
The universal $K3$ surface $\F_{14,1}$ is birational to the $\PP^{12}$-bundle $\mathfrak{X}$ over the moduli space $\mathfrak{H}_{\mathrm{scr}}$ of $3$-nodal septic scrolls $R\subseteq \PP^5$. A general fourfold $[X]\in \cC_{26}$ contains a two-dimensional family of such scrolls $R\subseteq X\subseteq \PP^5$. The space of such scrolls is isomorphic to the $K3$ surface associated to $X$.
\end{theorem}

\vskip 4pt

Theorem \ref{univk3} allows us to elucidate the structure of $\F_{14,1}$ even further and prove its rationality. We fix a  $3$-nodal septic scroll $R\subseteq \PP^5$ as above and denote its nodes by $p_1,p_2$, $p_3$. The curve $\Gamma_R\subseteq \GG(1,5)$ induced by the rulings of $R$  is a smooth rational septic curve admitting bisecant lines $L_1,L_2$ and $L_3$ in the Pl\"ucker embedding of $\GG(1,5)$. Precisely, $L_i$  parametrizes the lines passing through $p_i$ and contained in the $2$-plane $P_i$ spanned by the two rulings of $R$ that intersect at the node $p_i$, for $i=1,2,3$. Since $\Gamma_R$ spans a $7$-dimensional linear space in projective space $\PP^{14}$ containing $\GG(1,5)$, using Mukai's work \cite{M6} on realizing canonical genus $8$ curves as linear sections of the Grassmannian $\GG(1,5)$, it follows that the intersection
$\GG(1,5)\cdot \bigl\langle \Gamma_R \bigr\rangle$ is a semi-stable curve of genus $8$. We denote by $Q\subseteq \langle \Gamma_R \rangle=\PP^7$ the residual curve defined by the following equality:

\begin{equation}\label{liaison}
\GG(1,5)\cdot \bigl\langle \Gamma_R \bigr\rangle =\Gamma_R+L_1+L_2+L_3+Q.
\end{equation}

We shall establish in Lemmas \ref{intnumb} and \ref{quarticzero} that $Q$ is a smooth rational quartic curve and
$Q\cdot L_i=1$  for  $i=1,2,3$, as well as $Q\cdot \Gamma_R=3$. Therefore $Q$ is the curve of rulings of a quartic scroll $R_Q\subseteq \PP^5$, which contains three
rulings $\ell_1,\ell_2,\ell_3$, such that that $p_i\in \ell_i$ and $\ell_i\in P_i$ for $i=1,2,3$. In particular, $R_Q$ contains the three nodes of the septic scroll $R$.
We can show furthermore that $R_Q$ is smooth and isomorphic to $\FF_0$, see Theorem \ref{scrolltypes2}.

\vskip 5pt

The construction above can be reversed. Using the automorphism group  of the scroll $R_Q\subseteq \PP^5$, we fix three of its rulings $\ell_1, \ell_2, \ell_3\in \GG(1,5)$,
as well as points $p_i\in \ell_i$. We set $$\PP_i^3:=\bigl\{P_i\in \GG(2,5):\ell_i\subseteq P_i\bigr\},$$
for $i=1,2,3$, then define a map
$$\varkappa:\PP_1^3\times \PP_2^3\times \PP_3^3/\mathfrak{S}_3\dashrightarrow \mathfrak{H}_{\mathrm{scr}},$$
by reversing the above construction and using the decomposition (\ref{liaison}). Along with the fixed point $p_i$, each  $2$-plane $P_i\in \PP^3_i$ defines a line
$L_i\subseteq \GG(1,5)$ meeting the curve $Q$ at the point $\ell_i$. Precisely, $L_i$ is the line of lines in $P_i$ passing through the point $p_i$.
To the triple $(P_1,P_2,P_3)$ we associate the scroll $R\subseteq \PP^5$ whose associated curve of rulings $\Gamma_R$ is defined by the formula (\ref{liaison}). The above discussion indicates that $\varkappa$ is dominant. In fact more can be proved:

\begin{theorem}\label{rational2}
The moduli space of scrolls $\mathfrak{H}_{\mathrm{scr}}$ is birational to $\PP_1^3\times \PP_2^3 \times \PP^3_3/\mathfrak{S}_3$ and is thus rational.
\end{theorem}

Indeed, using the theorem on symmetric  functions, see \cite{Ma} or \cite{GKZ} Theorem 2.8 for a recent reference, all symmetric products of projective spaces are known to be rational. It is now clear that Theorem \ref{rational2} coupled  with Theorem \ref{univk3} implies that $\F_{14,1}$ is a rational variety.

\vskip 3pt

\noindent {\bf Acknowledgment:} We are most grateful to the referee, who carefully read the paper and whose many suggestions significantly improved the presentation and readability.

\section{$K3$ surfaces and cubic fourfolds}

We begin by setting some notation. Let $U\subseteq |\OO_{\PP^5}(3)|$ be the locus of smooth cubic fourfolds and set $$\cC:=U/PGL(6)$$ to be the
$20$-dimensional moduli space of cubic fourfolds. For an integer $d\equiv 0, 2 \ (\mbox{mod } 6)$, as pointed out in the Introduction, $\cC_d$ denotes the irreducible divisor of $\cC$ consisting of special cubic fourfolds of discriminant $d$. As usual, $\mathcal F_g$ is the irreducible $19$-dimensional moduli space of smooth polarized $K3$ surfaces $(S,H)$  of genus $g$, that is, with $H^2=2g-2$. We denote by $u:\F_{g,1}\rightarrow \F_g$  the universal $K3$ surface of genus $g$ in the sense of stacks. Each fibre $u^{-1}([S,H])$ is identified with the $K3$ surface $S$.

\vskip 4pt

Using the Hodge-theoretic similarity between $K3$ surfaces of genus $g=n^2+n+1$ and special cubic fourfolds of degree $2g-2$, Hassett \cite{H1} constructed a morphism of moduli spaces $$\varphi:\F_{n^2+n+2}\rightarrow \cC_{2(n^2+n+1)},$$ which is birational for $n\equiv 0,2 (\mbox{mod } 3)$, and of degree $2$ for $n\equiv 1  (\mbox{mod } 3)$ respectively. In particular, for $n=3$ there is a birational isomorphism of spaces of weight $2$ Hodge structures
$$\varphi:\F_{14}\stackrel{\cong}\longrightarrow \cC_{26},$$
that will be of use throughout the paper. At the moment, there is no geometric construction of the polarized $K3$ surface $\varphi^{-1}([X])$ associated to a general fourfold $[X]\in \cC_{26}$.

\vskip 4pt

We recall basic facts on Hilbert squares of $K3$ surfaces and refer to \cite{HT1} for a general reference on these matters.
Let $(S,H)$ be a polarized $K3$ surface with $\mbox{Pic}(S)=\mathbb Z\cdot H$ and $H^2=2g-2$. We denote by $S^{[2]}$ the Hilbert scheme of length two $0$-dimensional subschemes on $S$.
Then $H^2(S^{[2]}, \mathbb Z)$ is endowed with the \emph{Beauville-Bogomolov} quadratic form $q$. We denote by $\Delta\subseteq S^{[2]}$ the diagonal divisor consisting of
zero-dimensional subschemes  supported only at a single point and by $\delta:=\frac{[\Delta]}{2}\in H^2(S^{[2]},\mathbb Z)$ the reduced diagonal class. One has $q(\delta, \delta)=-2$. Note the canonical identification $$\Delta=\PP(T_S)=\cup \{\Delta_p:p\in S\},$$ where $\Delta_p$ is the rational curve consisting of those $0$-dimensional subschemes $\xi\in \Delta$ such that
$\mbox{supp}(\xi)=\{p\}$. We set $\delta_p:=[\Delta_p]\in H_2(S^{[2]}, \mathbb Z)$.

\vskip 4pt

For a curve $C\in |H|$ in the polarization class , we introduce the divisor
$$f_C:=\bigl\{\xi\in S^{[2]}:\mbox{supp}(\xi)\cap C\neq \emptyset\bigr\}$$
and set $f:=[f_C]\in H^2(S^{[2]},\mathbb Z)$.
If $p\in S$ is a general point, we also define the curve
$$F_p:=\bigl\{\xi=p+x\in S^{[2]}:x\in C\bigr\}$$
and set $f_p:=[F_p]\in H_2(S^{[2]},\mathbb Z)$.
The Beauville-Bogomolov form can be extended to a quadratic form on  $H_2(S^{[2]}, \mathbb Z)$, by
setting $q(\alpha, \alpha):=q(w_{\alpha}, w_{\alpha})$, with $w_{\alpha}\in H^2(S^{[2]},\mathbb Z)$ being the class characterized by
the property $\alpha\cdot u=q(w_{\alpha},u)$, for every $u\in H^2(S^{[2]},\mathbb Z)$. Here $\alpha \cdot u$ denotes the usual intersection product.

\vskip 4pt
One has the following decompositions, orthogonal with respect to $q$, both for the Picard group and for the group $N_1(S^{[2]},\mathbb Z)$ of $1$-cycles modulo numerical equivalence:
$$\mbox{Pic}(S^{[2]})\cong \mathbb Z\cdot f\oplus \mathbb Z\cdot \delta \  \mbox{ and } \ N_1(S^{[2]},\mathbb Z)\cong \mathbb Z\cdot f_p \oplus \mathbb Z\cdot  \delta_p.$$
We record, the more or less obvious relations:
\begin{equation}\label{intprod}
f\cdot f_p=2g-2, \ \delta\cdot \delta_p=-1, \ f\cdot \delta_p=0 \mbox{ and } \delta\cdot f_p=0.
\end{equation}

\vskip 4pt

Assume now that $X\subseteq \PP^5$ is a general special cubic fourfold of discriminant $26$ and let $[S,H]=\varphi^{-1}([X])\in \F_{14}$ be the associated  polarized $K3$ surface such that
\begin{equation}\label{ident1}
 S^{[2]}\cong F(X)\subseteq \GG(1,5)\hookrightarrow \PP^{14}.
 \end{equation}
Following \cite{BD}, let $\gamma_S:=[\OO_{S^{[2]}}(1)]$  be the hyperplane class of $\GG(1,5)$ restricted to the Hilbert square under the identification (\ref{ident1}).  Since $q(\gamma_S,\gamma_S)=6$, using (\ref{intprod}), it quickly follows that $$\gamma_S=2f-7\delta\in H^2(S^{[2]},\mathbb Z).$$

\begin{proposition}\label{scrolls1}
Suppose $[S,H]\in \F_{26}$ is a general element and let $R\subseteq S^{[2]}$ be an effective $1$-cycle such that $R\cdot \gamma_S=7$.
Then $R$ is one of the rational irreducible curves $\Delta_p$, for $p\in S$. In particular, $R$ is smooth.
\end{proposition}
\begin{proof}
Assume that $R$ is an effective $1$-cycle and write $[R]=af_p-b\delta_p\in N_1(S^{[2]},\mathbb Z)$. Since $7=R\cdot \gamma_S=52a-7b$, hence we can write $a=7a_1$, with $a_1\in \mathbb Z$,
and then $b=52a_1-1$. Using \cite{BM} Proposition 12.6, we have $q(R,R)\geq -\frac{5}{2}$. We obtain $39a_1^2-26a_1-1\leq 0$, and the only integer solution of this inequality is $a_1=0$,
therefore $[R]=\delta_p$.

Since $[R]\cdot \delta=-1$, it follows that $R\subseteq \Delta$. We claim that $R$ lies in one of the fibres of the $\PP^1$-bundle $\pi:\Delta=\PP(T_S)\rightarrow S$, which implies that $R=\Delta_p$, for some $p\in S$. Indeed, otherwise
$\pi(R)\equiv mH$, for some $m>0$. Accordingly, we write
$$mH^2=R\cdot \pi^{-1}(H)=R\cdot f=\delta_p\cdot f=0,$$
which is a contradiction.
\end{proof}

\begin{remark}
Unlike degree $26$, for  other values of $d$, a general $[X]\in \cC_d$ may contain several types of scrolls. For instance when $d=14$ and $\gamma_S=2f-5\delta$, the curves $\Delta_p$ with $p\in S$ correspond to quintic scroll, but $X$ also contains quartic scrolls corresponding to rational curves $R\subseteq F(X)$ with $[R]=3f_p-16\delta_p$. Note that $q(R,R)=-2$.
\end{remark}

\vskip 5pt

We now recall the correspondence between scrolls and rational curves in Grassmannians. Following for instance \cite{Dol} 10.4, we define a \emph{rational scroll} to be the image $R\subseteq \PP^n$ of a $\PP^1$-bundle $\pi:R'=\PP(\E)\rightarrow \PP^1$ under a map $\phi:R'\rightarrow \PP^n$ given by a linear subsystem of $|\OO_{\PP(\E)}(1)|$, thus sending the fibres of $\pi$ to lines in $\PP^n$.  Let $f_R:\PP^1\rightarrow \GG(1,n)$ be the map
$$f_R(t):=\bigl[\phi(\pi^{-1}(t))\bigr]$$ and denote by $\Gamma_R$ its image. Conversely, start with a non-degenerate map $f:\PP^1\rightarrow \GG(1,n)$,
then consider the pull-back under $f$ of the projectivization of tautological rank $2$ vector over $\GG(1,n)$, that is,
\begin{equation}\label{smoothscroll}
\Xi:=\Bigl\{(t,x):t\in \PP^1, x\in L_{f(t)}\Bigr\}\subseteq \PP^1\times \PP^n.
\end{equation}
Here $L_{f(t)}\subseteq \PP^n$ denotes the line whose moduli point in  $\GG(1,n)$ is precisely $f(t)$.

The projection $\pi_2:\Xi\rightarrow \PP^n$ is a finite map and its image is a scroll $R\subseteq \PP^n$ of degree
$$\mbox{deg}(\Gamma_R)=\mbox{deg} f^*\Bigl(\OO_{\GG(1,n)}(1)\Bigr).$$

\vskip 3pt

Throughout the paper, we interpret scrolls in terms of their associated curves of rulings. It will be useful to determine, using this language, when a scroll is smooth.

\begin{proposition}\label{rulcurve}
Let $R\subseteq \PP^n$ be a scroll which is not a cone and such that $\Gamma_R$ is a smooth rational curve in $\GG(1,n)$ which is not contained in a plane. Then there is a bijective correspondence between singularities of $R$ and
bisecant lines to $\Gamma_R$ lying on $\GG(1,n)$. In particular, if $\Gamma_R$ admits no bisecant lines contained in $\GG(1,n)$, then $R$ is smooth.
\end{proposition}
\begin{proof}
We consider the projection $\pi_2:\Xi\rightarrow R$ defined by (\ref{smoothscroll}). Then $\Xi$ is a smooth variety and the assumptions made on $R$ imply that  $\pi_2$ is a finite map. If a point $x\in R$ corresponds
to a singularity, then one of the two following possibilities occur: (i) the fibre $\pi_2^{-1}(x)$ consists of more than point, or (ii) the differential of $\pi_2$ at a point of $(t,x)\in \pi_2^{-1}(x)$ is not an isomorphism.

\vskip 3pt

In case (i), we choose distinct points $t_1, t_2\in \pi_1 \bigl(\pi_2^{-1}(x)\bigr)$. Denoting by $\ell_1:=f_R(t_1)$ and $\ell_2:=f_R(t_2)$ the rulings of $\Xi$ corresponding to these points, we observe that  $x\in \ell_1\cap \ell_2$. The set $L$ of lines in the $2$-plane $\langle \ell_1, \ell_2\rangle$ passing through $x$
is a line in $\GG(1,n)$ such that $\Gamma_R\cap L\supseteq \{\ell_1,\ell_2\}$, that is, $\Gamma_R$ possesses a secant line lying inside $\GG(1,n)$ in its Pl\"ucker embedding.
Note that $L$ is a genuine secant line in the sense that it meets the curve $\Gamma_R$ in two distinct points $\ell_1$ and $\ell_2$.
All lines lying inside $\GG(1,n)$ in its Pl\"ucker embedding  correspond to pencils of lines in a $2$-plane passing through a point in $\PP^n$. Thus conversely, when such a line meets
$\Gamma_R$ in two distinct points, these will correspond to two incident rulings of $R$. In particular $R$ is singular at their point of intersection.

\vskip 3pt

To deal with case (ii), we carry out a local calculation. Assume $(t_0,x)\in \Xi$ is a point at which the differential of $\pi_2$ is not an isomorphism. We set $\ell_0:=f_R(t_0)$ and
denote by $$p_{ij}(t)=a_i(t)b_j(t)-a_j(t)b_i(t), \ \ \mbox{ where } 0\leq i<j\leq n$$ the Pl\"ucker coordinates of the curve $\Gamma_R$ in a neighborhood of $\ell_0$,
where $a(t)=(a_0(t), \ldots, a_n(t))$ and $b(t)=(b_0(t), \ldots, b_n(t))$.

\vskip 3pt

In local coordinates, the map $\pi_2$ is given by $\PP^1\times \mathbb C\ni ([\lambda,\mu],t)\mapsto \bigl[(\lambda a_i(t)+\mu b_i(t))\bigr]=:x$. By direct calculation, the condition that $(d\pi_2)_{(t_0,x)}$ is not an isomorphism is equivalent to
$$b'(t_0)\wedge a(t_0)=0\in \bigwedge ^2 \mathbb C^{n+1}.$$

Setting $a_i:=a_i(t_0)$, $b_i:=b_i(t_0)$, $a_i':=a_i'(t_0)$ and
$b_i':=b_i(t_0)$, we then observe that the Pl\"ucker coordinates of a point on the tangent line $\mathbb T_{\ell_0}(\Gamma_R)\subseteq \PP^{{n+1\choose 2}-1}$ are given by
$$a_ib_j-a_jb_i+\mu\bigl(a_i'b_j+a_ib_j'-a_j'b_i-a_jb_i'\bigr)=b_j(a_i+\mu a_i')-b_i(a_j+\mu a_j'),$$
for some scalar $\mu$. It follows that the tangent line to $\Gamma_R$ at $\ell_0$  is contained in $\GG(1,n)$. The argument being reversible, we finish the proof.
\end{proof}

\vskip 5pt

The scrolls $R\subseteq \PP^n$ we consider most of the time have at worst \emph{non-normal nodal singularities} $x\in R$,  corresponding to the case $|\phi^{-1}(x)|=2$.
The tangent cone of $R$ at $x$ is isomorphic to the union of two $2$-planes in $\PP^4$ meeting in one point. According to Proposition \ref{rulcurve}, to each such singularity corresponds a line in the Pl\"ucker embedding of $\GG(1,n)$ meeting $\Gamma_R$ in two distinct points.

\vskip 3pt

Suppose now that $R\subseteq X\subseteq \PP^5$ is a rational scroll with isolated nodal singularities contained in a cubic fourfold. Using  the \emph{double point formula} \cite{Ful} 9.3
applied to the map $\phi:R'\rightarrow X$, we find the number of singularities of $R=\phi(R')$:
\begin{equation}\label{doublepoints}
D(\phi)=R^2-6h^2-K_R^2-3h\cdot K_R+2\chi_{\mathrm{top}}(R).
\end{equation}

When $[X]\in \cC_{26}$, assuming that $A(X)=\langle h^2,[R]\rangle$, where $h^2\cdot [R]=\mbox{deg}(R)=7$, necessarily $R^2=25$. From formula
(\ref{doublepoints}), we compute $D(\phi)=3$, that is, if $R$ has only (isolated) improper nodes, then it is $3$-nodal.

\vskip 4pt

Before stating our next result, we recall that $\mm_{0}(F(X),7)$ denotes the space of stable maps $f:C\rightarrow F(X)$, from a nodal curve $C$ of genus zero such that
$\mbox{deg}(f^*(\OO_{F(X)}(1))=7$. We denote by $\cM_{0}(F(X),7)$ the open sublocus consisting of maps with source $\PP^1$ and denote by $\mm_7(X)$
the closure of $\cM_{0}(F(X),7)$ inside $\mm_0(F(X),7)$.

\begin{corollary}\label{k326}
Let $[X]\in \cC_{26}$ a general special fourfold of discriminant $26$ and $[S,H]\in \F_{26}$ its associated $K3$ surface. Then there is an isomorphism $S\cong \mm_7(X)$.
\end{corollary}
\begin{proof}
Using the identification $S^{[2]}\cong F(X)$, we define the map $j:S\rightarrow \mm_7(X)$, by setting $j(p):=\Delta_p\subseteq F(X)$. All points in the image
of $j$ consist of embedded smooth rational curves  $\PP^1\stackrel{\cong}\hookrightarrow  \Delta_p$ and we identify $\Delta_p$ with the corresponding
map $\PP^1\hookrightarrow F(X)$. In a neighborhood of this map, the moduli space $\mm_0(F(X),7)$ is locally isomorphic to the Hilbert scheme
of septic rational curves on $F(X)$.
\vskip 3pt

The tangent space of $\mm_7(X)$ at the point $[\Delta_p]$ is canonically isomorphic to $H^0(N_{\Delta_p/F(X)})$. Using the following exact sequence on $\Delta_p\cong \PP^1$
$$0\longrightarrow N_{\Delta_p/\Delta}\longrightarrow N_{\Delta_p/F(X)}\longrightarrow \OO_{\Delta_p}(\Delta)\longrightarrow 0,$$
since $N_{\Delta_p/\Delta}=\OO_{\Delta_p}^{\oplus 2}$ and $\OO_{\Delta_p}(\Delta)=\OO_{\Delta_p}(-1)$, we compute $N_{\Delta_p/F(X)}=\OO_{\Delta_p}^{\oplus 2}\oplus \OO_{\Delta_p}(-1)$.
It follows that $H^1(\Delta_p, N_{\Delta_p/F(X)})=0$, hence the obstruction space for deformations vanishes and
$$\mbox{dim } T_{[\Delta_p]}(\mm_0(F(X),7)=h^0(\Delta_p, N_{\Delta_p/F(X)})=2.$$
We conclude that $[\Delta_p]$ is a smooth point of expected dimension of $\mm_7(X)$, for every $p\in S$.

\vskip 3pt
Furthermore, $j$ is injective, because for distinct points $p,q\in S$, since $\Delta_p\cap \Delta_q=\emptyset$, the associated scrolls $R_p$ and $R_q$ share no rulings.
We finally observe that $j$ is an immersion. Indeed, for each $p\in S$, we have the identification
$\Delta_p=\PP\Bigl(T_p(S)\oplus T_p(S)/T_p(S)\Bigr)$, the quotient being given by the diagonal embedding. Thus the differential $dj(p)$ is essentially the identity map,
via the identification $\PP(T_S)\cong \bigcup_{p\in S} \PP\bigl(N_{\Delta_p/\Delta}\bigr)$. Since according to Proposition \ref{scrolls1}, we have that $\cM_0(F(X),7)\subseteq \mbox{Im}(j)$, we can
conclude the proof.
\end{proof}

\section{Nodal septic scrolls and cubic fourfolds}

In this section we study in more detail the moduli space $\mathfrak{H}_{\mathrm{scr}}$ of $3$-nodal septic scrolls that will be used to parametrize the universal $K3$ surface of degree $26$.
We fix once and for all the smooth septic scroll
$$R':=S_{3,4} \hookrightarrow \PP^8,$$
given as the image of the map $\phi_{|4\ell-3E|}$ on the Hirzebruch surface $\FF_1=\mbox{Bl}_{\mathpzc{o}}(\PP^2)$. We denote by $h:R'\rightarrow \PP^1$ the map induced by the linear system $|\ell-E|$. The fibres of $h$ are pairwise disjoint lines in $\PP^8$. Equivalently, we consider the vector bundle on $\PP^1$
$$\G=\OO_{\PP^1}(3)\oplus \OO_{\PP^1}(4)$$ and then $R'\cong \PP(\G)$. One has the canonical identification between space of sections:
$$H^0\bigl(R',\OO_{R'}(1)\bigr)\cong H^0\bigl(\PP(\G), \OO_{\PP(\G)}(1)\bigr)\cong H^0\bigl(\PP^1, \G\bigr).$$

Later, when computing the dimension of the parameter space of $3$-nodal septic scrolls, we shall make use of the basic fact $$\mbox{dim } \mbox{Aut}(R')=\mbox{dim } \mbox{Aut}(\FF_1)=6.$$
Every smooth septic scroll in $\PP^8$ is obtained from $R'$ by applying a linear transformation of $\PP^8$. In particular, the Hilbert scheme of septic scrolls in $\PP^8$ has dimension
equal to $$\mbox{dim } PGL(9)-\mbox{dim } \mbox{Aut}(R')=80-6=74.$$

\vskip 4pt

Using coordinates in $\PP^8$, if $\PP^3_{y_0, \ldots, y_3}\subseteq \PP^8$ is the linear span of the twisted cubic $E$ corresponding to the exceptional divisor on $\FF_1$ and
$\PP^4_{x_0,\ldots, x_4}\subseteq \PP^8$ is the linear span of a rational quartic curve linearly equivalent to $\ell$, then the ideal of $R'$ in $\PP^8$ is given by the following
determinantal condition, see
for instance \cite{Ha} Lecture 9:
$$\mbox{rk}\begin{pmatrix}
x_0 & x_1& x_2& x_3 & y_0 & y_1 &y_2\\
x_1 & x_2 & x_3& x_4& y_1 &y_2 &y_3 \\
\end{pmatrix}
\leq 1.$$

The secant variety $\mbox{Sec}(R')\subseteq \PP^8$ is also  determinantal, with equations given by the $3\times 3$ minors  of the following 1-\emph{generic} matrix:

$$\mbox{rk} \begin{pmatrix}
x_0 & x_1& x_2& y_0& y_1\\
x_1 & x_2 & x_3& y_1 & y_2 \\
x_2 & x_3 & x_4 & y_2 & y_3\\
\end{pmatrix}
\leq 2
$$
It follows from \cite{CJ} Lemma 3.1 that, as expected, $\mbox{Sec}(R')$ is $5$-dimensional. Furthermore, applying e.g. \cite{Ei} Corollary 3.3, it follows that the singular
locus of $\mbox{Sec}(R')$ coincides with the scroll $R'$.

\vskip 4pt

\begin{lemma}\label{genproj}
 Let $a_1,a_2,a_3\in \mathrm{Sec}(R')$ be general points and set $\Lambda:=\langle a_1,a_2,a_3\rangle\in \GG(2,8)$. The image $R$ of the projection $\pi:R'\rightarrow \PP^5$ with
 center $\Lambda$ has three non-normal nodes corresponding to the three bisecant lines passing through $a_1, a_2$ and $a_3$ and no further singularities.
\end{lemma}
\begin{proof} The chosen  points $a_1,a_2,a_3$ can be assumed to lie in $\mbox{Sec}(R')-(R'\cup \mbox{Tan}(R'))$. Since $\mbox{dim } \mbox{Sec}(R')=5$, by using the \emph{Trisecant lemma},
see for instance \cite{CC} Proposition 2.6, it follows that the scheme-theoretic
intersection of $\mbox{Sec}(R')$ with $\Lambda$ consists only of the points $a_1,a_2,a_3$. In particular, $\Lambda\cap R'=\emptyset$, hence the projection $\pi=\pi_{\Lambda}:R'\rightarrow R$
is a regular morphism. Furthermore, each point $a_i$ lies on a unique bisecant line $\langle x_i, y_i\rangle$, where $x_i$ and $y_i$ are distinct points of $R'$, for $i=1,2,3$.

\vskip 3pt

Suppose now that for  $x,y\in R'$, one has $\pi(x)=\pi(y)$. This happens  if and only if $\langle x,y\rangle \cap \Lambda\neq \emptyset$, hence $\emptyset \neq \langle x,y\rangle\cap \Lambda
\subseteq \{a_1,a_2,a_3\}$ and then
necessarily $\{x,y\}=\{x_i,y_i\}$, for $i\in \{1,2,3\}$. Since $\Lambda\cap \mbox{Tan}(R')=\emptyset$, it follows that the differential of $\pi$ is everywhere injective.
To summarize, the only singularities of $R$ are the three non-normal nodes $\pi(x_i)=\pi(y_i)$, for $i=1,2,3$.
\end{proof}

We now fix a general projection $\pi=\pi_{\Lambda}:R'\rightarrow \PP^5$ as in Lemma \ref{genproj}. We denote by $p_i$ the three singularities of the image scroll
$R$.  The map $\pi_{_{\Lambda}}$ is defined by the 6-dimensional  subspace
$V := H^0(\PP^8, \mathcal I_{\Lambda / \PP^8}(1))$ of $H^0(\PP^1, \mathcal G)$. To give $\Lambda$ amounts to specifying $V \subseteq H^0(\PP^1, \mathcal G)$. Since $\Lambda \cap R'=\emptyset$ ,  it follows that the evaluation map $\mathrm{ev}_V: V \otimes \mathcal O_{\mathbf P^1} \to \mathcal G$ is surjective. Hence $\mathrm{ev}_V$ defines a morphism
$$
f: \PP^1 \rightarrow \GG(1,5).
$$
This map is induced by the ruling of the image scroll $R$, that is, $f_R=f$ is the map given by $f_R(t):=\bigl[\pi(h^{-1}(t))\bigr],$ for $t\in \PP^1$. Set $\Gamma_R:=\mbox{Im}(f_R)$.

\begin{proposition}\label{transv3}
For a general choice of the $3$-secant plane $\Lambda$ to $\mathrm{Sec}(R')$, the following hold:
\begin{enumerate}
\item $\mathrm{dim } \langle p_1, p_2,  p_3 \rangle=2$.
\item $\langle p_1, p_2,  p_3\rangle \cap R = \{p_1, p_2, p_3\}$.
\end{enumerate}
\end{proposition}
\begin{proof} It suffices to consider a codimension $2$ general linear section $Z \subseteq R'\subseteq \PP^8$. Then $Z$ is a smooth $0$-dimensional scheme supported at seven distinct points $x_1,y_1,x_2,y_2,x_3,y_3$ and $z$, spanning a  $6$-dimensional linear space in $\PP^8$. In particular, $z$ does not lie in the $5$-plane spanned by the points $\{x_i,y_i\}_{i=1}^3$ and no line intersecting the lines  $\langle x_1, y_1 \rangle$,
$\langle x_2, y_2 \rangle$, $\langle x_3, y_3 \rangle$ exists. Pick general points $a_i \in \langle x_i, y_{i} \rangle$, for $i=1,2,3$. Then the projection $\pi_{_{\Lambda}}$ defined by the plane $\Lambda = \langle a_1, a_2, a_3 \rangle$ satisfies both conditions  (i) and (ii).
\end{proof}

\vskip 4pt

For a projection $\pi_{\Lambda}$ satisfying the assumptions of Lemma \ref{genproj}, the map $f_R:\PP^1\rightarrow \GG(1,5)$ is an embedding, for $\Lambda$ intesects no ruling of $R'$.
We record the conclusion of Proposition \ref{rulcurve} for a scroll $R$ as above:

\begin{proposition}\label{bisecant}
The rational curve $\Gamma_R\subseteq \GG(1,5)$ admits three secant lines that lie in $\GG(1,5)$.
Conversely, a rational septic curve $\Gamma\subseteq \GG(1,5)$ having three secant lines lying in $\GG(1,5)$ is the curve of rulings of a $3$-nodal septic scroll
in $\PP^5$.
\end{proposition}

We establish a couple of properties concerning the linear system of cubic fourfolds containing a $3$-nodal septic scroll:

\begin{proposition}\label{vanish}
The following statements hold for a general $3$-nodal septic scroll $R\subset \PP^5$:
$$ (i) \ \ \mathrm{dim } |\mathcal{I}_{R/\PP^5}(3)|=12 \ \ \ \mbox{ and } \ \ \  (ii)\  \ \ H^1(\PP^5, \I_{R/\PP^5}(3))=0.$$
\end{proposition}
\begin{proof}
Recall that $R$ is the image of a projection $\pi=\pi_{\Lambda}:R'\rightarrow R$ with center $\Lambda$, and denote by $p_1,p_2,p_3\in R$ the three (non-normal) singularities of $R$ and
by $\{x_i,y_i\}=\pi^{-1}(p_i)$, for $i=1,2,3$. By Proposition \ref{transv3}, the points $p_1, p_2$ and $p_3$ are in general linear position in $\PP^5$ and thus impose independent conditions
on cubic hypersurfaces, that is, $H^1(\PP^5, \I_{\mathrm{Sing}(R)/\PP^5}(3))=0$.

By passing to cohomology in the short exact sequence
$$0\longrightarrow \I_{R/\PP^5}(3)\longrightarrow \I_{\mathrm{Sing}(R)/\PP^5}(3)\longrightarrow \I_{\mathrm{Sing}(R)/R}(3)\longrightarrow 0,$$
we write the following exact sequence:
$$0\longrightarrow H^0(\I_{R/\PP^5}(3))\longrightarrow H^0(\I_{\mathrm{Sing}(R)/\PP^5}(3))\longrightarrow H^0(\I_{\mathrm{Sing}(R)/R}(3))\longrightarrow H^1(\I_{R/\PP^5}(3))\longrightarrow 0.$$
Clearly $h^0\bigl(\PP^5,\I_{\mathrm{Sing}(R)/\PP^5}(3)\bigr)={8\choose 3}-3=53$. Furthermore,  we have the following identification of linear systems:
\begin{equation}\label{idlin}
\pi^*\Bigl(\bigl|\I_{\mathrm{Sing}(R)/R}(3)\bigr|\Bigr)=\Bigl|\I_{\{x_1,y_1,x_2,y_2, x_3,y_3\}/R'}\bigl(12\ell-9E\bigr)\Bigr|.
\end{equation}
The scroll $[R]\in \mathfrak{H}_{\mathrm{scr}}$ is obtained as a general projection like in Lemma \ref{genproj}. In particular,
the points $\{x_i,y_i\}_{i=1}^3\subseteq  R'$ are general as well, hence impose independent conditions on the linear system $|12\ell-9E|$ on $R'$.
Using the identification (\ref{idlin}), we compute:
$$h^0\bigl(R, \I_{\mathrm{Sing}(R)/R}(3)\bigr)=h^0\bigl(R',\OO_{R'}(12\ell-9E)\bigr)-6=h^0(\PP^2, \OO_{\PP^2}(12))-{10\choose 2}-6=40.$$
Therefore $h^0\bigl(\PP^5,\I_{R/\PP^5}(3)\bigr)=13$ if and only if $h^1\bigl(\PP^5, \I_{R/\PP^5}(3)\bigr)=0$. This last statement
can be proved via a simple \emph{Macaulay} calculation by choosing the points $a_1, a_2, a_3$ randomly in the variety $\mbox{Sec}(R')$ whose
equations have been given explicitly.
\end{proof}

\vskip 3pt

\begin{remark}
It is possible to realize the rational curve $\Gamma_R$ inside the linear system $|\OO_R(1)|$ as follows.
Recall that we have denoted by $\phi:\FF_1\hookrightarrow \PP^8$ the embedding whose image is the smooth scroll $R'$. In $|4\ell-3E|\cong \PP^8$, we consider the space of reducible hyperplane sections:
$$\Bigl\{A'+L': A'\in |3\ell-2E|, \ L'\in |\ell-E|\Bigr\}.$$
Note that $L'$ is a ruling of $R'$, whereas $A'\subseteq \PP^8$ is a sextic with $\langle A' \rangle =\PP^6$ and with $L'\cdot A'=1$. In the linear system $|3\ell-2E|$ there exists a
\emph{unique} sextic $A_0'$ such that $\Lambda \subseteq \langle A_0'\rangle\subseteq \PP^8$. The curve $A_0'$  corresponds to the unique curve in
the linear system
$$\Bigl|\I_{\{x_1,y_1,x_2,y_2,x_3,y_3\}/R'}(3\ell-2E)\Bigr|$$ on $R'$.
Indeed, $x_i,y_i\in A_0'$, therefore $a_i\in \langle x_i,y_i\rangle \subseteq \langle A_0'\rangle$, for $i=1,2,3$. It then follows that $\Lambda=\langle a_1,a_2,a_3\rangle \subseteq \langle A_0'\rangle$. The projection $A_0:=\pi(A_0')\subseteq \PP^5$ is
a sextic curve on $R$ passing through the nodes $p_1, p_2, p_3$. One identifies $\Gamma_R$ with $A_0$ via the map $L \mapsto L\cdot A_0$.
\end{remark}
\vskip 4pt


We denote by $\mathcal{H}_{\mathrm{scr}}$ the Hilbert scheme of $3$-nodal septic scrolls in $R\subseteq \PP^5$ and set
$$\mathfrak{H}_{\mathrm{scr}}:=\mathcal{H}_{\mathrm{scr}}/PGL(6).$$
We regard $\mathfrak{H}_{\mathrm{scr}}$ as the coarse moduli space of $3$-nodal septic scrolls.

\begin{proposition}\label{isom1}
The parameter space $\mathfrak{H}_{\mathrm{scr}}$ is  birationally isomorphic to
 $\mathrm{Sym}^3\bigl(\mathrm{Sec}(R')\bigr)/\mathrm{Aut}(R')$. In particular, $\mathfrak{H}_{\mathrm{scr}}$ is a unirational $9$-dimensional variety.
\end{proposition}

\begin{proof}
 We identify $\mbox{Aut}(R')$ with the group consisting of linear automorphisms $\sigma\in PGL(9)$ such that $\sigma(R')=R'$. Every  $\sigma\in \mbox{Aut}(R')$ clearly invariates
 $\mbox{Sec}(R')$. Since $\mbox{Sing} (\mbox{Sec}(R'))=R'$, conversely, every automorphism $\sigma\in PGL(9)$ invariating $\mbox{Sec}(R')$ belongs actually to
 $\mbox{Aut}(R')$. One has a birational action of $\mbox{Aut}(R')$ on $\mbox{Sym}^3 \bigl(\mbox{Sec}(R')\bigr)$ given by
 $$\sigma \langle a_1,a_2,a_3\rangle:=\langle \sigma(a_1),\sigma(a_2),\sigma(a_3)\rangle,$$
 for $\sigma\in \mbox{Aut}(R')$ and $a_1,a_2,a_3\in \mbox{Sec}(R')$. We can now define a birational morphism
 $$\vartheta:\mathrm{Sym}^3\bigl(\mathrm{Sec}(R')\bigr)/\mathrm{Aut}(R')\dashrightarrow \mathfrak{H}_{\mathrm{scr}}, \ \mbox{ by setting}$$
 $$\Lambda:=\langle a_1,a_2,a_3\rangle \mapsto \pi_{\Lambda}(R') \ \mathrm{ mod }\ PGL(6),$$
 where $\pi_{\Lambda}:\PP^9\dashrightarrow \PP^5$ is a projection of center $\Lambda$. The assignment is clearly $\mbox{Aut}(R')$-invariant, hence $\vartheta$ is well-defined. Applying Lemma \ref{genproj}, we obtain that $\vartheta$ is a birational isomorphism.

 \vskip 3pt
 The secant variety $\mbox{Sec}(R')$ being a $\PP^1$-bundle over the rational variety $\mbox{Sym}^2(R')$ is unirational. This implies that $\mbox{Sym}^3\bigl(\mbox{Sec}(R')\bigr)$
 and thus $\mathfrak{H}_{\mathrm{scr}}$ are unirational as well.
\end{proof}

Over the Hilbert scheme  $\mathcal{H}_{\mathrm{scr}}$ we consider the universal family of scrolls:
$$
\begin{CD}
{\mathcal{H}_{\mathrm{scr}}} @<p_1<< {\mathcal{Y}_{\mathrm{scr}}} @>p_2>> {\PP^5} \\
\end{CD}.
$$


\vskip 4pt

We introduce the incidence correspondence between cubic fourfolds of discriminant $26$ and nodal septic scrolls in $\PP^5$:
$$
\begin{CD}
{|\OO_{\PP^5}(3)|} @<<< {\X:=\PP\Bigl((p_1)_*\bigl(\I_{\mathcal{Y}_{\mathrm{scr}}/\mathcal{H}_{\mathrm{scr}}\times \PP^5}\otimes p_2^*\OO_{\PP^5}(3)\bigr)\Bigr)} @>>> {\mathcal{H}_{\mathrm{scr}}} \\
\end{CD}
$$
The group $PGL(6)$ acts on the entire diagram. By quotienting out this action, if we set $\mathfrak{X}:=\mathcal{X}/PGL(6)$, we obtain two projections:
$$
\begin{CD}
{\cC_{26}} @<\pi_1<< {\mathfrak{X}} @>\pi_2>> {\mathfrak{H}_{\mathrm{scr}}} \\
\end{CD}
$$

The $21$-dimensional variety $\mathfrak{X}$ being a $\PP^{12}$-bundle over the unirational variety $\mathfrak{H}_{\mathrm{scr}}$ is unirational as well.
A general scroll $[R]\in \mathfrak{H}_{\mathrm{scr}}$ has precisely $3$ non-normal nodes. Checking that a general cubic fourfold $X\supseteq R$ is smooth, reduces to a standard Macaulay calculation. Applying (\ref{doublepoints}), we obtain that the lattice $A(X)$ contains a $2$-dimensional lattice $\langle h^2,[R]\rangle$ of discriminant $26$, therefore the map $\pi_1$ is well-defined.
Proposition \ref{scrolls1} implies  $\mbox{dim } \pi_1^{-1}([X])\leq 2$, for all $[X]\in \cC_{26}$, hence $\mathfrak{X}$ dominates $\cC_{26}$. In fact one can prove something more precise  and  establish in the process Theorem \ref{univk3}.

\begin{theorem} The incidence correspondence $\mathfrak{X}$ is birational to the universal $K3$ surface $\F_{14,1}$.
\end{theorem}
\begin{proof}
We define a map $\theta:\mathfrak{X} \rightarrow \F_{14,1}$ as follows. We start with a pair $[X,R]\in \mathfrak{X}$ and denote by $f_R:\PP^1\rightarrow F(X)$ the rational curve of rulings described in
Proposition \ref{bisecant}. Denoting by $[S,H]:=\phi^{-1}([X])\in \F_{14}$ the polarized $K3$ surface provided by the identification (\ref{ident1}), applying Proposition \ref{scrolls1}, there exists a uniquely determined point $p\in S$ such that $\Delta_p=\Gamma_R$.

\vskip 3pt

The map $\theta$ is clearly generically injective. Since both $\mathfrak{X}$ and $\F_{14,1}$ are irreducible varieties of the same dimension $21$, it follows that $\theta$ is birational. In particular, in
the isomorphism $S\cong \mm_7(X)$ constructed in Corollary \ref{k326}, the general point on both sides corresponds to a septic scroll $R\subseteq X$ which is $3$-nodal and has no further singularities.
\end{proof}

\vskip 4pt



\section{The rationality of $\F_{14,1}$}
In this section, using in an essential way the characterization given in Proposition \ref{bisecant} of the rational curves $\Gamma_R$ of rulings of $3$-nodal scrolls $R\subseteq \PP^5$,  we show that the universal $K3$ surface of genus $14$ is rational.

We begin by recalling the structure of the moduli space of curves of genus $8$. Consider the Grassmannian $\GG(1,5)\subseteq \PP^{14}$ in its Pl\"ucker embedding. Denote by
$$\mathfrak{M}_8:=\GG \Bigl(7,\PP\Bigl(\bigwedge^2 \mathbb C^6\Bigr)\Bigr)/ PGL(6)$$ the space of codimension $7$ linear sections of $\GG(1,5)$. Mukai \cite{M6} has shown that the map
$$\mathfrak{M}_8\dashrightarrow \mm_8,
$$
sending a general $7$-plane $[\PP(V)\hookrightarrow \PP^{14}] \in \mathfrak{M}_8$ to the intersection $[\GG(1,5)\cdot \PP(V)] \in \overline {\mathcal M}_8$ viewed as a canonical curve of genus $8$, is a birational isomorphism. For more details on how to extend Mukai's isomorphism over parts of the boundary of $\mm_8$, see also \cite{FV2}.


\vskip 4pt

Recall that we introduced in Section 3 the smooth septic scroll $R'\cong \FF_1 \subseteq \PP^8$, then considered a singular scroll $R\subseteq \PP^5$, defined as the image of a linear projection
$\pi_{_{\Lambda}}:R'\rightarrow \mathbf P^5 $ whose center is a general plane $\Lambda \subset \PP^8$, which is $3$-secant to $\mbox{Sec}(R')$.
We denote by $p_1, p_2, p_3$ the three nodes of $R$ and $\{x_i,y_i\}=\pi^{-1}(p_i)$. As explained in the Introduction, $P_i\subseteq \PP^5$ denotes
the $2$-plane spanned by the rulings of $R$ passing through $p_i$, for $i=1,2,3$. The line $$L_i\subseteq \GG(1,5)\subseteq \PP^{14}$$ parametrizes the lines in the
plane $P_i$ passing through the point $p_i$. If $\Gamma=\Gamma_R\subseteq \GG(1,5)$ is the curve of rulings associated to $R$ introduced in Proposition \ref{bisecant}, then $L_i$ meets $\Gamma$ in two distinct points. We keep this notation throughout this section.

\vskip 4pt

Due to the results of the previous section,  our strategy is now to describe the family
$$
\mathcal{U} \subseteq \mbox{Hom}(\PP^1,\GG(1,5))
$$
of smooth rational septic curves $\Gamma_R\subseteq \GG(1,5)$ carrying three bisecant lines contained in $\GG(1,5)$. From Proposition \ref{bisecant} it follows that  $\mathcal{U}$ is birational to the  Hilbert scheme $\mathcal{H}_{\mathrm{scr}}$ of $3$-nodal septic scrolls in $\PP^5$. Then we show that
the quotient  $\mathcal{U}/PGL(6)$ is rational. Since $\mathcal{U}/PGL(6)$ is birational to $\mathfrak{H}_{\mathrm{scr}}$ and,
as proven in Theorem \ref{univk3}, the universal $K3$ surface of genus $14$ is a $\PP^{12}$-bundle over $\mathfrak{H}_{\mathrm{scr}}$, its rationality will follow.

\vskip 4pt

The nodal curve
$
\Gamma + L_1 + L_2 + L_3 \subseteq \langle \Gamma \rangle \cdot  \mathbf G(1,5)
$
has arithmetic genus 3. It follows from Mukai's work \cite{M1} that the intersection  $\langle \Gamma \rangle \cdot \mathbf G(1,5)$ is a canonical curve of genus $8$, provided (i) it is proper and reduced and (ii) $\mbox{dim } \langle \Gamma \rangle =7$. Using the surjectivity of the period map for polarized $K3$ surfaces of genus $8$, we shall show that both assumptions (i) and (ii) are satisfied. Granting both (i) and (ii) for the moment, we consider the canonically embedded curve in $\langle \Gamma \rangle=\PP^7$, pictured also below:
\begin{equation}\label{linint}
C:=\langle \Gamma \rangle \cdot \mathbf G(1,5) = Q + \Gamma + L_1 + L_2 + L_3.
\end{equation}
\begin{figure}[h]
\centering
  \includegraphics[width=2.3in]{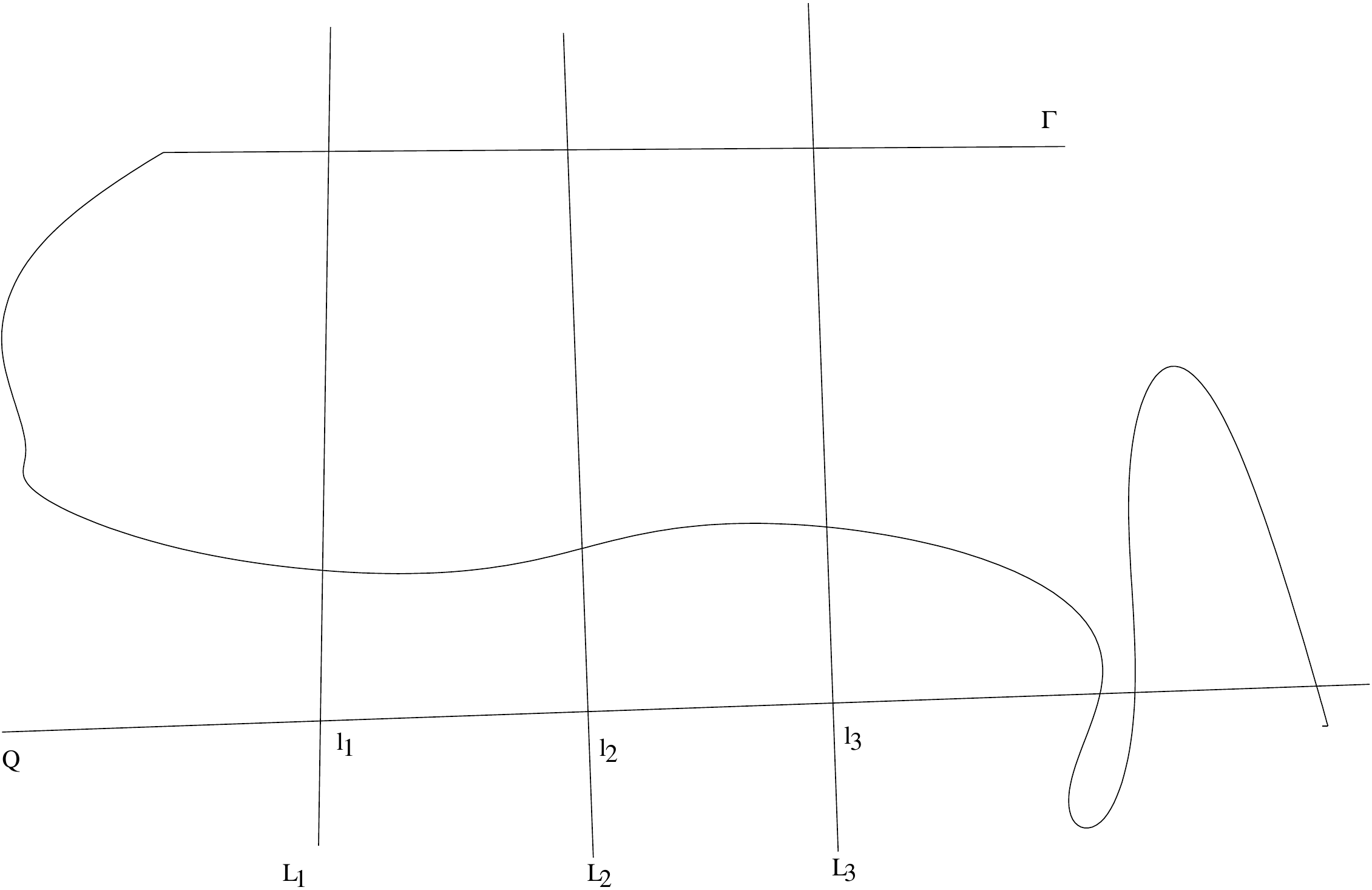}
  \caption{The canonical curve $C=\Gamma+Q+L_1+L_2+L_3.$}
  \label{fig:cover}
\end{figure}

Bertini's Theorem implies that a general $8$-dimensional space  $\langle \Gamma \rangle\subseteq \PP^8\subseteq \PP^{14}$ cuts out on $\mathbf G(1,5)$ a smooth 2-dimensional linear section
$T$, see also \cite{Ve}, Propositions 3.2 and 3.3. By the adjunction formula, $T\hookrightarrow \PP^8$ is a smooth $K3$ surface (of genus $8$) polarized by $\mathcal O_T(C)$.
We now describe the Picard lattice of $T$:

\begin{lemma}\label{intnumb}
One has the following intersection products on $T$:
$$
Q^2=-2, \ Q \cdot \Gamma = 3,\  \  Q\cdot  L_i = 1, \  \ \  \Gamma \cdot L_i  = 2, \   \ L_i \cdot L_j = -2\delta_{ij},
 \ \ \mbox{ for } i,j=1,2,3.
$$
\end{lemma}
\begin{proof} The generality assumptions ensure that $L_i$ and $L_j$ are disjoint lines, for $i\neq j$. Else, if $L_i\cap L_j\neq \emptyset$, then $\langle p_i,p_j\rangle \subseteq P_i\cap P_j\subseteq \PP^5$. It follows that the four rulings of $R'$ passing through the points $x_i,y_i,x_j,y_j$ respectively,  span a $6$-dimensional space in $\PP^8$, which is impossible for
$$h^0\Bigl(R',\OO_{R'}(1)(-4(\ell-E))\Bigr)=h^0(R',\OO_{R'}(E))=1,$$
where recall that $\ell, E\in \mbox{Pic}(R')$ denote the line class and the exceptional divisor respectively. This implies that there exists a unique hyperplane in $\PP^8$ containing the four
rulings, therefore they must span a $7$-dimensional linear space.

Since $L_i^2=-2$, by intersecting (\ref{linint}) with $L_i$, we obtain
$Q\cdot L_i=1$. Furthermore $7=\Gamma \cdot C$ and since $\Gamma^2=-2$, we obtain $\Gamma \cdot Q=3$. Finally, $C\cdot Q=\mbox{deg}(Q)=4$, therefore $Q^2+\Gamma\cdot Q+3=4$, implying $Q^2=-2$ and thus finishing the proof.
\end{proof}

In particular $Q\subseteq \langle T\rangle =\PP^8$ is a reduced, connected quartic curve of arithmetic genus zero. Since
$C - Q \equiv \Gamma + L_1 + L_2 + L_3$, we obtain $h^0(T, \OO_T(C-Q))=4$.  The next lemma summarizes the situation.

\begin{lemma}\label{quarticzero} The span $\langle Q \rangle$ is 4-dimensional and $Q$ is a connected nodal quartic curve with $p_a(Q) = 0$. \end{lemma}

In fact, we shall construct a $K3$ surface $T$, such that the curve $Q$ described in Lemma \ref{quarticzero} is actually smooth.

\vskip 3pt

To establish the validity of the assumptions (i) and (ii) and thus the existence of the special $K3$ surface $T$, we use Hodge theory. We consider the following sublattice
\begin{equation}\label{lattice}
\mathbb L := \mathbb Z \cdot [Q] \oplus  \mathbb Z \cdot [\Gamma] \oplus \mathbb Z \cdot [L_1] \oplus \mathbb Z \cdot [L_2] \oplus \mathbb Z\cdot [L_3]
\end{equation}
generated by the $(-2)$ classes corresponding to  $Q, \Gamma, L_1, L_2$ and $L_3$ respectively, and with intersection pairing as given in Lemma \ref{intnumb}.  We  invoke  the
surjectivity of the period map for $K3$ surfaces. The rank $5$ lattice $\mathbb L$ is even and has signature $(1,4)$. Applying \cite{Mo} Corollary 2.9, there exists a smooth
$K3$ surface $T$, such that $\mbox{Pic}(T)\cong \mathbb L$. We define the following class on $T$
$$C:=\Gamma+Q+L_1+L_2+L_3.$$
The genus zero curves $\Gamma, Q, L_1, L_2, L_3\subseteq T$ cannot have multiple components, for that would make $\mbox{Pic}(T)$ larger than $\mathbb L$, therefore they are all smooth,
rational curves on $T$.

\begin{lemma}\label{veryample}
The linear system $|\OO_T(C)|$ is very ample.
\end{lemma}
\begin{proof} We use Reider's Theorem \cite{R}, which, in the case of $K3$ surfaces, had been proven before in \cite{SD}. It suffices to show that there exists no curve $E$ on $T$ with $E^2=0$ and $E\cdot C\in \{1,2\}$, nor a curve $F$ on $T$
with $F^2=-2$ and
$F\cdot C=0$. We prove the first statement, the second follows similarly. Assuming there is such a curve $E$, we express it as an integral combination
$E\equiv x\Gamma+yQ+z_1L_1+z_2L_2+z_3L_3$ of the generators of $\mbox{Pic}(T)$. If $C\cdot E=1$, we obtain
$$-15x^2-12xy-5y^2+2x+y=z_1^2+z_2^2+z_3^2.$$
By comparing the signs of the two sides, one concludes that this equation has no integral solutions. The case $C\cdot E=2$ is similar. Finally, if $
F\equiv x\Gamma+yQ+z_1L_1+z_2L_2+z_3L_3$ is a $(-2)$-curve with $C\cdot F=0$, we obtain
$$-15x^2-12xy-5y^2+1=z_1^2+z_2^2+z_3^2,$$
which implies $x=y=0$ and, say $z_2=z_3=0$ and then $z_1=1$. Thus $F=L_1$, but $C\cdot L_1=1$, hence this case does not appear.
We conclude that $C$ is very ample.
\end{proof}

We show that the $K3$ surface $T$ constructed in Lemma \ref{veryample} is a linear section of $\GG(1,5)$. In particular, Mukai's results \cite{M6} will apply for
its hyperplane section $C$.

\begin{proposition}\label{grassm}
The $K3$ surface $T$ carries a globally generated rank two vector bundle $T$ with $\mathrm{det}(T)=\OO_T(C)$, providing an embedding $T\hookrightarrow \GG(1,5)$ such that
$$\langle T\rangle \cdot \GG(1,5)=S.$$
\end{proposition}

\begin{proof}
We use \cite{M7} and need to show that the polarized $K3$ surface $(T,\OO_T(C))$ is \emph{Brill-Noether general}, that is, for all pairs of line bundles $M,N$ on $T$ such that
$M\otimes N=\OO_T(C)$, one has $h^0(T,M)\cdot h^0(T,N)<h^0(T,C)$. Under these circumstances, it is shown in \emph{loc.cit.} that $T$ carries a rigid, globally generated, stable rank $2$
vector bundle $E$ with $h^0(T,E)=6$ and
$\mbox{det}(E)=\OO_T(C)$, inducing a map $\varphi_{E}:T\rightarrow \GG(1,5)$. Reasoning along the lines of \cite{M7} Theorem 3.10, the $K3$ surface $T$ is then a linear section of
$\GG(1,5)$ in its Pl\"ucker embedding, that is,
$T=\GG(1,5)\cdot \langle T\rangle.$

\vskip 4pt

To establish the Brill-Noether generality of $(T,\OO_T(C))$, we use for instance \cite{GLT} Lemma 2.8.  It suffices to show that in the lattice $\mathbb L$ there exists no vector $D$
such that $D^2=2$ and $D\cdot C\in \{7,6\}$, nor is there a vector $D$ with $D^2=0$ and $D\cdot C\leq 4$.

\vskip 4pt

We treat in detail only the first case, the remaining ones being similar. We write
$$D\equiv x\Gamma+yQ+z_1L_1+z_2L_2+z_3L_3.$$ The conditions $D^2=2$ and $D\cdot C=7$ translate into the equalities $z_1+z_2+z_3+7x+4y=7$ and $-15x^2-5y^2-12xy+14x+7y+1=z_1^2+z_2^2+z_3^2\geq 0$. It is elementary to see that there are no integral solutions.
\end{proof}

Using Proposition \ref{grassm}, we conclude that the intersection (\ref{linint}) corresponding to a general curve $\Gamma_R\in \mathcal{U}$ corresponds to a
semistable canonical curve of genus $8$.

\vskip 5pt

It will be useful to have a criterion for determining when the  curve $\Gamma$ spans a space of maximal possible dimension in the Pl\"ucker space $\PP^{14}\supseteq \GG(1,5)$. To that end, recall that the Pl\"ucker embedding of the dual Grassmannian
$\GG(1,5)^{\vee}=\mathbf G(3,5) \hookrightarrow \bigl(\PP^{14}\bigr)^{\vee} $
assigns to a point $p\in \GG(1,5)^{\vee}$ corresponding to a $3$-plane $\PP^3_p\subseteq \PP^5$ the Schubert cycle
$$\sigma_p:=\bigl\{\ell\in \GG(1,5): \ell\cap \PP^3_p\neq \emptyset\bigr\}.$$ Note that  $\dim \langle \Gamma \rangle + 1 = \codim \langle \Gamma \rangle^{\perp}$. Setting
$$
W^1(\Gamma) := \mathbf G(3,5) \cap \langle \Gamma \rangle^{\perp} = \bigl\{p \in \GG(3,5): \Gamma \subseteq \sigma_p \bigr\},
$$
for dimension reasons, the next lemma follows immediately:
\begin{lemma} Assume $W^1(\Gamma)$ is finite. Then $\dim \langle \Gamma \rangle = 7$.
\end{lemma}

\vskip 4pt

Keeping the previous notation, let $f_R: \PP^1 \to \GG(1,5)$ be a sufficiently general element of $\mathcal U$ and set again $\Gamma=\Gamma_R$. Then under the assumption  $R'=S_{3,4}$, we can prove that:
\begin{theorem}\label{finiterulings}
 The set $W^1(\Gamma)$ is finite. In particular $\dim \langle \Gamma \rangle = 7$ and $\Gamma$ is a rational normal septic curve.
\end{theorem}
\begin{proof} If $p \in W^1(\Gamma)$, then $\mathbf P^3_p$ contains an integral curve intersecting each line of $R$. Its strict transform by $\pi_{_{\Lambda}}:R'\rightarrow R$ is an
integral section  $A$ of the ruled surface $R'$. Set $d:=\deg(A)$, hence $A\equiv (d-3)\ell-(d-4)E\in \mbox{Pic}(\FF_1)$. Clearly $\langle A\rangle \subseteq \pi_{\Lambda}^{-1}(\PP^3_p)$,
implying  $\mbox{dim } \langle A\rangle \leq 6$.

\vskip 3pt

Let $\mathbf I_{A}:=|H-A|$ be the linear system of hyperplanes in $\PP^8$ containing the curve $A\subseteq R'$. By direct calculation, we find
$\dim(\mathbf I_A)$ $=$ $\dim \vert H - A \vert$ $=$ $7 - d\geq 1$
and $\dim |A|= 2d - 6$. It follows that $3\leq d\leq 6$. Recalling that $V=H^0(\PP^8, \mathcal{I}_{\Lambda/\PP^8}(1))$,  the condition
$$
\dim (\mathbf PV \cap \mathbf I_A) \geq 1
$$
is equivalent to the condition that the  curve $\pi_{_{\Lambda}}(A)$ be contained in a $3$-space $\mathbf P^3_p$. For $3 \leq d \leq 6$ let $\mathbf G(7-d,|H|)$ denote the Grassmannian of $(7-d)$-subspaces  of $|H|\cong \PP^8$ and introduce the $(2d-6)$-dimensional variety
$$
\mathbf{S}_d :=\Bigl\{\mathbf I_{A'}\in \GG(7-d, |H|): A' \in |(d-3)\ell-(d-4)E|\Bigr\}.
$$
For an integer $k \geq 1$, we consider the Schubert cycle
$$
\sigma^k_V := \bigl\{\mathbf I \in \mathbf G(7-d,|H|): \dim (\mathbf PV \cap \mathbf I) \geq k \bigr\}.
$$
The cycle $\sigma^k_V \cdot \mathbf{S}_d$ is finite for $k = 1$ and empty for $k \geq 2$, provided the intersection is proper. By Kleiman's transversality of a general translate this is true for a general translate of $\sigma^k_V$ in $\mathbf G(d-7,|H|)$, that is, for a general choice of $\Lambda$ (or equivalently, of $V$). Hence $W^1(\Gamma)$ is finite.
\end{proof}

\begin{remark} The theorem above fails for rational septic scrolls in $\mathbf P^8$ containing sections of degree $d\leq 2$, that is, for the scrolls $S_{a,7-a}$, where $a\neq 3$.
\end{remark}

We turn to the smooth residual rational curve $Q\subseteq \GG(1,5)$ defined by (\ref{linint}). Let
$$
R_Q \subseteq \PP^5
$$
be the quartic scroll whose rulings are parametrized by the curve $Q$.

\begin{lemma} $R_Q$ is a non-degenerate smooth rational normal scroll in $\mathbf P^5$. \end{lemma}

\begin{proof} First, observe that  $R_Q$ cannot be a cone. Let us assume $R_Q$ is a cone of vertex $v\in \PP^5$. Then $\langle Q \rangle\cong \PP^4\subseteq \mathbf G(1,5)$  parametrizes  the lines passing through $v$. This is a contradiction because  $\langle Q \rangle \subseteq \langle \Gamma \rangle \cdot \mathbf G(1,5) = C$. Now assume that $R_Q$ is contained in a hyperplane $H\subseteq \PP^5$. Then $Q$ is contained in the
Grassmannian $\GG_H:=\GG(1,H)\subseteq \GG(1,5)$ of lines of  $H$. Since $K_{\GG_H}=\OO_{\GG_H}(-5)$, we observe that, by adjunction, the curvilinear sections of $\GG_H$ are curves of
arithmetic genus $1$. Because of this fact and since $\mbox{deg}(\GG_H)=5$,  it follows that
$$\langle Q \rangle \cdot \GG_H = Q + L \subseteq C,$$ where $L$ is a bisecant line to $Q$. But the only line components in $C$ are $L_1, L_2, L_3$ and none of them is bisecant to $Q$. Via Proposition \ref{rulcurve}, the same argument shows that the scroll $R_Q$ has no incident rulings, therefore $R_Q$ is smooth.
\end{proof}

\begin{lemma}\label{scrolltypes}
The scroll $R_Q$ contains no other lines except the ruling parametrized by $Q$.
\end{lemma}
\begin{proof}    Assume $R_Q$ contains a line $\ell_0$ not parametrized by a point of $Q$. We prove that this implies that $W^1(\Gamma)$ is not finite, thus contradicting Theorem \ref{finiterulings}.  Consider the family $G$ of codimension $1$ Schubert cycles $\sigma_p$ defined by a 3-space $\PP^3_p\supseteq \ell_0$. Note that $G \cong \GG(1,3)$. We have $G \subseteq \langle Q \rangle^{\perp}$. Since $\langle Q \rangle \subseteq \langle \Gamma \rangle$, we also have $\langle \Gamma \rangle^{\perp}\subseteq \langle Q \rangle^{\perp}$. Counting dimensions it follows $\dim (G \cap \langle \Gamma \rangle^{\perp}) \geq 1$,  which implies that $W^1(\Gamma)$ is not finite. \end{proof}

There are two types of smooth quartic scrolls in $\PP^5$, namely
$S_{1,3}=\mathbf P\bigl(\mathcal O_{\mathbf P^1}(1) \oplus \mathcal O_{\mathbf P^1}(3)\bigr)$ and $S_{2,2}= \PP\bigl(\mathcal O_{\mathbf P^1}(2) \oplus \mathcal O_{\mathbf P^1}(2)\bigr)$.
The latter case is characterized by the property that every line contained in the scroll is a ruling. Lemma \ref{scrolltypes} implies the following:

\begin{theorem}\label{scrolltypes2}
Let $\Gamma\subseteq \GG(1,5)$ be a smooth septic rational curve corresponding to a general element of $\mathcal U$ and $Q\subseteq \GG(1,5)$ the residual quartic curve. Then $R_Q$ is
isomorphic to $S_{2,2}$.
\end{theorem}

To summarize, to a general rational curve $\Gamma=\Gamma_R \in \mathcal{U}$, we associated the quartic scroll $R_Q$, equipped with three rulings $\ell_1, \ell_2, \ell_3$ corresponding to
the points $L_i \cdot Q\in \GG(1,5)$, for $i=1,2,3$. Each ruling $\ell_i$ passes through the node $p_i$ of the scroll $R$ and is contained in the $2$-plane $P_i$ whose existence is
established in Proposition \ref{bisecant}.

\vskip 6pt

To prove the rationality of $\mathfrak{H}_{\mathrm{scr}}$ and thus that $\F_{14,1}$, we reverse this construction. We denote by $\mathcal{V}$  the variety classifying elements
$(R_Q, p_1, p_2, p_3)$, where $R_Q\subseteq \PP^5$ is a smooth quartic scroll isomorphic to $S_{2,2}$ and $p_i\in R_Q$ for $i=1,2,3$.

\begin{lemma}\label{stab}
The $PGL(6)$-stabilizer of a general point $(R_Q,p_1,p_2,p_3)\in \mathcal{V}$ is trivial. In particular, $PGL(6)$ acts transitively on $\mathcal{V}$.
\end{lemma}
\begin{proof}
The automorphism group of $S_{2,2}\cong \FF_0$ is the semidirect
product of $PGL(2)\times PGL(2)$ with $\mathbb Z/2\mathbb Z$. The last factor corresponds to the automorphism $u\in \mbox{Aut}(\FF_0)$ permuting the two factors. In particular,
$\mbox{Aut}(S_{2,2})$ is $6$-dimensional. This implies that the space $\mathcal{V}$ has dimension
$$\mbox{dim } PGL(6)-\mbox{dim } \mbox{Aut}(S_{2,2})+3 \mbox{dim}(R_Q)=35= \mbox{dim } \ PGL(6).$$ Choose general points $p_i=(a_i, b_i)\in \FF_0\cong S_{2,2}$, with $a_i\neq b_i$,
for $i=1,2,3$. Up to the action of $u\in \mbox{Aut}(\FF_0)$, the stabilizer $\mbox{Stab}_{PGL(6)}(R_Q,p_1,p_2,p_3)$ corresponds to pairs of automorphism
$(\sigma_1,\sigma_2)\in PGL(2)\times PGL(2)$, such
that $\sigma_1(a_i)=a_i$ and $\sigma_2(b_i)=b_i$. Thus $\sigma_1=\sigma_2=1$. The points $p_i$ not lying on the diagonal of $\FF_0$, the automorphism $u$ does not fix any of them, thus
the stabilizer in question is trivial. Since $\mathcal{V}$ and $PGL(6)$ have the same dimension, this also implies the trasitivity of the $PGL(6)$-action on $\mathcal{V}$, as claimed.
\end{proof}

\vskip 3pt

We can thus start by fixing once and for all the quartic scroll $R_Q$. Precisely, we embed the surface $\FF_0:=\PP^1\times \PP^1$ in $\PP^5$ via the linear system $|\OO_{\FF_0}(1,2)|$ and denote by $$R_0\subseteq \PP^5$$ the image quartic scroll. The rulings on $R_0$ are the elements of the linear system $|\OO_{\FF_0}(0,1)|$. Let $Q_0\subseteq \GG(1,5)$  be the curve of rulings of $R_0$.
We then fix three points in $\FF_0$, for instance
$$
\mathpzc{o}_1:= \bigl([1:0],[0:1]\bigr), \ \mathpzc{o}_2:= \bigl([0:1],[1:0]\bigr) \ \mbox{ and } \mathpzc{o}_3:= \bigl([1:1],[-1:-1]\bigr),
$$
which we identify with their images in $R_0$.  As explained in Lemma \ref{stab}, the stabilizer subgroup $G$ of $PGL(6)$ fixing both $R_0$ as well as the \emph{set} $\{\mathpzc{o}_1, \mathpzc{o}_2, \mathpzc{o}_3\}$ is isomorphic
to the subgroup of $PGL(2)\times PGL(2)$ fixing the set $\{\mathpzc{o}_1, \mathpzc{o}_2, \mathpzc{o}_3\}$. Therefore $G = \mathfrak S_3$.

\vskip 3pt

For $i=1,2,3$, we denote by  $\ell_i$ the ruling of $R_0$ passing through the point $\mathpzc{o}_i$. Then, let $\PP_i^3$ be the projective space consisting of $2$-planes $\Pi_i \subseteq \mathbf P^5$ containing the line $\ell_i$. Giving a plane $\Pi_i$ is equivalent to specifying a  line $L_i\subseteq \GG(1,5)$  in the Pl\"ucker embedding of the Grassmannian.
Note that $L_i$ meets $Q_0$ transversally at precisely one point, namely $\ell_i\in \GG(1,5)$.

\vskip 5pt

We introduce a rational map
$$
\varkappa: \mathbf P^3_1 \times \mathbf P^3_2 \times \mathbf P^3_3/\mathfrak{S}_3  \dasharrow \mathfrak{H}_{\mathrm{scr}}
$$
defined as follows. To a triple of planes $(\Pi_1,\Pi_2,\Pi_3)$, we attach the lines $L_1, L_2, L_3\subseteq \GG(1,5)$. Since $Q_0\subseteq \GG(1,5)$ is a smooth rational quartic curve,
in the Pl\"ucker embedding we have that $\langle Q_0\rangle \cong \PP^4$. Attaching one general $1$-secant line to $Q_0$ increases the dimension of the linear span of the union by one,
therefore by attaching three general $1$-secant lines, we have
 $$\langle Q_0+L_1+L_2+L_3\rangle \cong \PP^7 \subseteq \PP^{14}.$$  We write
 $$\langle Q_0+L_1+L_2+L_3\rangle \cdot \GG(1,5)=Q_0+L_1+L_2+L_3+\Gamma,$$
 where $\Gamma$ is a degree $7$ curve. Applying Lemma \ref{intnumb}, it follows that $\Gamma$ is a rational curve and $\Gamma\cdot L_i=2$, for $i=1,2,3$. We denote by $\ell_i'$ and $\ell_i''$ the intersection points $L_i\cdot \Gamma$. From Proposition \ref{bisecant} it follows that
 that the scroll $R:=R_{\Gamma}$ induced by $\Gamma$ is $3$-nodal, with nodes given by the intersection $\ell_i'\cap \ell_i''$ taken in the $2$-plane $\Pi_i$. We set
 $$\varkappa(\Pi_1+\Pi_2+ \Pi_3):=[R].$$

\vskip 3pt

We conclude the proof of the rationality of the Hilbert scheme of $3$-nodal scrolls in $\PP^5$:

\vskip 5pt

\noindent \emph{Proof of Theorem \ref{rational2}.}
We first observe that $\varkappa$ is well-defined. To that end, we choose the polarized $K3$ surface $(T, \OO_T(C))$ constructed in Propositions
\ref{veryample} and \ref{grassm} and we keep the notation used there. Applying Theorem \ref{scrolltypes2}, the residual quartic rational curve $Q\subseteq \GG(1,5)$ parametrizes the rulings
of a quartic scroll $R_Q\subseteq \PP^5$, which is isomorphic to $S_{2,2}$. Applying Lemma \ref{stab}, there exists a unique automorphisms $\sigma\in PGL(6)$ such that $\sigma(R_Q)=R_0$ and
$\sigma(p_i)=\mathpzc{o}_i$, for $i=1,2,3$. Set $\sigma(P_i)=:\Pi_i\in \PP^3_i$ and then $\varkappa(\Pi_1+\Pi_2+\Pi_3)=[R_{\Gamma}]$.

\vskip 3pt

To finish the proof it suffices to observe that $\varkappa$ is generically injective. A general septic curve $\Gamma \in \mathcal{U}$  corresponding to a $3$-nodal septic scroll
$[R_{\Gamma}]\in \mathfrak{H}_{\mathrm{scr}}$ has precisely $3$ bisecant lines lying in
$\GG(1,5)$. Giving $\Gamma$ determines its linear span $\langle \Gamma \rangle$, hence the set $\{L_1, L_2, L_3\}$ as well.
\hfill $\Box$

\section{The unirationality of the universal $K3$ surface of genus at most $12$}

We denote by $\F_{g,n}$ the universal $n$-pointed $K3$ surface of genus $g$. Thus $\F_{g,n}$ is an irreducible variety of dimension $19+2n$. Similarly, one can consider the
universal Hilbert scheme of $0$-dimensional cycles of length $n$, that is, $u^{[n]}:\F_{g}^{[n]}\rightarrow \F_g$. We also introduce the notation $\cC_{g,n}:=\cM_{g,n}/\mathfrak{S}_n$ for the degree $n$ universal symmetric product over $\cM_g$, where the symmetric group $\mathfrak{S}_n$ acts by permuting the marked points.

\vskip 3pt

The aim of this short last section is to point out how Mukai's results
determine the birational type of $\F_{g,n}$ and that of $\F_{g}^{[n]}$ for small $g$, and thus put our Theorem \ref{main1} better into context:

\begin{theorem}\label{kodaira} The following results on the Kodaira dimension of $\F_{g,n}$ hold:
 \begin{enumerate}
 \item $\F_{g,g+1}$ is unirational for $g\leq 10$.
 \item $\F_{11,1}$ is unirational. The Kodaira dimension of both $\F_{11,11}$ and $\F_{11}^{[11]}$ equals $19$.
 \end{enumerate}
\end{theorem}
\begin{proof}

For $g\leq 5$, the general $K3$ surface of genus $g$ is a complete intersection in a projective space and the result follows easily. For details,  see the
table after Theorem 1.10  in \cite{M7}.

For $6\leq g\leq 10$, Mukai \cite{M1} has constructed a rational homogeneous variety
$V_g\subseteq \PP^{N_g}$, where $N_g=g+\mbox{dim}(V_g)-2$, such that the general $K3$ surface of genus $g$ is obtained as a general linear section
$S=V_g\cap \Lambda_g$, where $\Lambda_g\subseteq \PP^{N_g}$ is a $g$-dimensional plane, with the polarization being the one induced by $\OO_{\PP^{N_g}}(1)$. Moreover, one has the following  birational isomorphism, see \cite{M1} Corollary 0.3:
$$\F_g\stackrel{\cong}\dashrightarrow \GG(g,N_g)/ \mbox{Aut}(V_g).$$
These results imply the existence of a dominant map $\chi_g:V_g^{g+1}\dashrightarrow \F_{g,g+1}$ given by
$$\chi(x_1,\ldots,x_{g+1}):=\bigl[V_g\cap \langle x_1,\ldots, x_{g+1}\rangle, x_1, \ldots, x_{g+1}\bigr].$$ This proves that $\F_{g,g+1}$ (and hence $\F_{g,n}$ for $n\leq g+1$) is unirational in this range.

\vskip 3pt


For $g=11$, we use \cite{M8}, where it is shown that  a general curve $[C]\in \cM_{11}$ lies on a \emph{unique} $K3$ surface $C\subseteq S$ as a hyperplane section, with $\mbox{Pic}(S)=\mathbb Z\cdot C$. This implies the existence of a rational map
$\chi_{n}:\cM_{11,n}\dashrightarrow \F_{11,n}$ defined by $$\chi_{n}([C,x_1,\ldots,x_{n}]):=[S,x_1,\ldots, x_{n}].$$ The map $\chi_n$ is dominant for $n\leq 11$ and a birational isomorphism
for $n=11$. Indeed, in this last case, given an embedded $K3$ surface $S\stackrel{|H|}\hookrightarrow \PP^{11}$ and general points $x_1, \ldots,x_{11}\in S$, the hyperplane $\langle x_1,\ldots,x_{11}\rangle\cong \PP^{10}$ cuts out a canonical genus $11$ curve $C$ on $S$, which comes equipped with the marked points $x_1, \ldots, x_{11}$. By quotienting the action of the symmetric group $\mathfrak{S}_{11}$, the map $\chi_{11}$ induces a birational isomorphism between the universal symmetric product $\cC_{11,11}$ and $\F_{11}^{[11]}$. Now we use \cite{FV1} Theorem 0.5. Both varieties $\cM_{11,11}$ and $\cC_{11,11}$ have Kodaira dimension $19$, hence we conclude.

\vskip 4pt

We now pass on to the universal $K3$ surface $\F_{11,1}$. To that end we define a rational map
$$\vartheta:\cM_{10,2}\dashrightarrow \F_{11,1},$$ associating to a $2$-pointed curve
$[C,p_1,p_2]\in \cM_{10,2}$, the unique $K3$ surface $S$ of genus $11$ containing the curve $[X:=C/p_1\sim p_2]$ obtained from $C$ by identifying $p_1$ and $p_2$. To show that $\vartheta$ is well-defined, that is, Mukai's construction \cite{M8} can be also carried out for the $1$-nodal curve $[X]\in \mm_{11}$, we use \cite{CLM} Proposition 4.4. Observe that the $K3$ surface $S$ has a distinguished point corresponding to the image of the singularity of $X$.  The map $\vartheta$ is clearly dominant, for in each linear system on a $K3$ surface, the $1$-nodal curves fill-up a divisor. The unirationality of $\F_{11,1}$ now follows from that of $\cM_{10,2}$, which can be established in a variety of ways, see for instance \cite{BCF} Theorem B.

\end{proof}

\begin{remark}
It is claimed incorrectly in \cite{L} Table 3, that $\cM_{11,n}$ is unirational for $n\leq 10$. The argument sketched in \emph{loc.cit.} only establishes the uniruledness of $\cM_{11,n}$ when $n\leq 10$, precisely using the map $\chi_n:\cM_{11,n}\rightarrow \F_{11,n}$, which is birationally a $\PP^{11-n}$-bundle. But this argument alone offers no indications concerning the birational nature of the base variety $\F_{11,n}$. One can establish partial results on the birational nature of $\F_{11,n}$, for $n\leq 10$. For instance,
it is shown in \cite{Ve} that the universal product $\cC_{11,6}$ is unirational, which implies that $\F_{11}^{[6]}$ is unirational as well.
\end{remark}

\begin{remark} Mukai \cite{M4} gives an explicit orbit space realization over a projective space for the universal $K3$ surface $\F_{13,1}$. The unirationality of $\F_{13,1}$ thus follows. Presumably, a similar argument works for genus $12$, when $\F_{12}$ is known to be birational to a $\PP^{13}$-bundle over the rational moduli space $\mathcal{MF}_{22}$ of Fano $3$-folds $V_{22}\subseteq \PP^{13}$, see again \cite{M1}.
\end{remark}

\begin{remark}\label{19}
Since $u:\F_{g,1}\rightarrow \F_g$ is a morphism fibred in Calabi-Yau varieties, by Iitaka's easy addition formula $\kappa(\F_{g,1})\leq \mbox{dim}(\F_g)=19$, in particular,
$\F_{g,1}$ is never of general type. Furthermore, by \cite{K}, we also write $\kappa(\F_{g,1})\geq \kappa(\F_g)$. In particular, when $\F_g$ is of general type, then $\kappa(\F_{g,1})=19$.
\end{remark}

\end{document}